\documentclass[11pt,twoside]{article}
\usepackage{mathrsfs}
\usepackage{amsmath}
\usepackage{amsthm}
\input amssymb.sty
\usepackage{amsfonts}
\usepackage{amssymb}
\usepackage{latexsym}

\date{\empty}
\allowdisplaybreaks[4] \footskip=15pt

\numberwithin{equation}{section} \theoremstyle{plain}
\newtheorem*{thm*}{Main Theorem}
\newtheorem{theorem}{Theorem}[section]
\newtheorem{corollary}[theorem]{Corollary}
\newtheorem*{corollary*}{Corollary}

\newtheorem*{claim*}{Claim}
\newtheorem{lemma}[theorem]{Lemma}
\newtheorem*{lemma*}{Lemma}
\newtheorem{proposition}[theorem]{Proposition}
\newtheorem*{proposition*}{Proposition}
\newtheorem{remark}[theorem]{Remark}
\newtheorem*{remark*}{Remark}

\newtheorem*{example*}{Example}

\newtheorem*{question*}{Question}

\newtheorem*{definition*}{Definition}

\begin{document}

\begin{center}
{\large  \bf Drazin Invertibility of Product and Difference of Idempotents in a Ring}\\
\vspace{0.8cm} {\small {\bf  Jianlong Chen\footnote{Corresponding author. Email: jlchen@seu.edu.cn}, \ \ \ Huihui Zhu} \\
Department of Mathematics, Southeast University, Nanjing 210096, China.}
\end{center}

\bigskip
{\bf  Abstract:}  \leftskip0truemm\rightskip0truemm In this paper, several equivalent conditions on the Drazin invertibility of product and difference of idempotents are obtained in a ring. Some results in Banach algebra are extended to the ring case.
 \bigskip


\section { \bf Introduction}

~~~~Throughout this paper, $R$ denotes an associative ring with unity $1$. Recall that an element $a\in R$ is said to be Drazin invertible if there exists $b\in R$ such that
\begin{center}
$ab=ba$, $bab=b$, $a^k=a^{k+1}b$
\end{center}
for some positive integer $k$. The element $b$ above is unique if it exists and denoted by $a^D$. The such least $k$ is called the Drazin index of $a$, denoted by ${\rm ind}(a)$. If ${\rm ind}(a)=1$, then $b$ is called group inverse of $a$ and denoted by $a^ \#$. By $R^D$ we mean the set of all Drazin invertible elements of $R$.

Drazin inverses are closely related to some regularity of rings. It is well known that $a$ is Drazin invertible if and only if $a$ is strongly $\pi$-regular (i.e., $a^n\in a^{n+1}R \cap Ra^{n+1}$ for some nonnegative integer $n$) if and only if $a^m$ is group invertible for some positive integer $m$. By [9], we know that if $a$ is Drazin invertible, then $a^m$ is Drazin invertible for any positive integer $m$. Hence, $a$ is Drazin invertible if and only if $a^2$ is Drazin invertible. In [11, 12], Koliha and Rako\v{c}evi\'{c} studied the invertibility of the difference and the sum of idempotents in a ring and proved that $p-q$ is invertible if and only if $1-pq$ and $p+q$ are invertible for any idempotents $p$ and $q$. Since then, this topic attracted broad attention. Many authors extended the ordinary invertibility to the (generalized) Drazin invertibility. For example, Deng and Wei [8] proved that if $p$, $q$ are idempotents in $\mathscr{A}=\mathcal{B}(X)$, the ring of all bounded linear operators in a complex Banach space $X$. Then, $(1)$ $p-q\in \mathscr{A}^D$ if and only if $p+q\in \mathscr{A}^D$ if and only if $1-pq\in \mathscr{A}^D$; $(2)$ $pq-qp\in \mathscr{A}^D$ if and only if $pq+qp \in \mathscr{A}^D$ if and only if $pq\in \mathscr{A}^D$ and $p-q\in \mathscr{A}^D$. These results are generalized to Banach algebra in [10]. More results on the Drazin invertibility of sum, difference and product of idempotents can be found in [3-8, 10, 13].

In this paper, we consider the Drazin invertibility of $p-q$, $pq$, $pq-qp$ (commutator) and $pq+qp$ (anti-commutator), where $p$ and $q$ are idempotents in a ring. Some results in [3, 4] for Banach algebra are extended to the ring case.
\section {\bf Key lemmas}
~~~~In this section, we begin with some elementary and known results which will be useful in section 3.
\begin{lemma}$([9])$
$(1)$  Let $a\in R^D$. If $ab=ba$, then $a^Db=ba^D$,\\
 $(2)$ Let $a$, $b\in R^D$ and $ab=ba=0$. Then $(a+b)^D=a^D+b^D$.
\end{lemma}
\begin{lemma} $([14])$ Let $a, b\in R^D$ and $ab=ba$. Then $ab\in R^D$ and $(ab)^D=b^Da^D=a^Db^D$.
\end{lemma}

\begin{lemma} $([2])$ Let $a, b\in R$ and $ab\in R^D$. Then $ba\in R^D$ and $(ba)^D=b((ab)^D)^2a$.
\end{lemma}

\begin{lemma}$([1])$ Let $a, b\in R$. Then $1-ab\in R^D$ if and only if $1-ba\in R^D$.
\end{lemma}

\begin{lemma}
Let $a, b\in R^D$ and $p^2=p\in R$. If $ap=pa$ and $bp=pb$, then $ap+b(1-p)\in R^D$ and
\begin{center}
$(ap+b(1-p))^D=a^D p+b^D(1-p)$.
\end{center}
\end{lemma}
\begin{proof} Since $p^2=p$, we have $p^D=p$. Thus, $a,$ $b,$ $p\in R^D$. Note that $ap=pa$ and $bp=pb$. We obtain $(ap)^D=a^Dp$ and $(b(1-p))^D=b^D(1-p)$ by Lemma 2.2. As $apb(1-p)=b(1-p)ap=0$, according to Lemma 2.1(2), it follows that $(ap+b(1-p))^D=a^D p+b^D(1-p)$.
\end{proof}

\begin{lemma}  Let $a \in R$, $p^2=p\in R$, $b=pa(1-p)$ and $c=(1-p)ap$. The following statements are equivalent:\\
 $(1)$ $b+ c \in R^D$,\\
 $(2)$ $bc\in R^D$,\\
 $(3)$ $b-c\in R^D$.
\end{lemma}

\begin{proof}
$(1)\Rightarrow (2)$ Since $(b+ c)^2=(bc+cb)$ and $b+ c \in R^D$, we have $ bc+cb \in R^D$. Let $x=(bc+cb)^D$. As $p(bc+cb)=(bc+cb)p$, we obtain that $px=xp$ by Lemma 2.1(1). Next, we prove that $(bc)^D=pxp$.

Since $x=(bc+cb)x^2$, we get
\begin{center}
$pxp=p(bc+cb)x^2p=bcx^2p=bc(px^2p)=bc(pxp)^2$.
\end{center}
 By $(bc+cb)x=x(bc+cb)$, we obtain $p(bc+cb)xp=px(bc+cb)p$. It follows that $bc(pxp)=(pxp)bc$.

 Because $(bc+cb)^{n+1}x=(bc+cb)^n$ for some $n$, we have
 $$((bc)^{n+1}+(cb)^{n+1})x=(bc)^n+(cb)^n.$$

Multiplying the equation above by $p$ on two sides yields
\begin{center}
$p(bc)^{n+1}xp=p(bc)^np$,
\end{center}
i.e., $(bc)^{n+1}pxp=(bc)^n$. So, $bc$ is Drazin invertible and $(bc)^D=pxp$.

$(2) \Rightarrow (1)$ According to Lemma 2.3, $bc\in R^D$ is equivalent to $cb\in R^D$. Note that $bc \cdot cb=cb\cdot bc=0$. We have $(b+ c)^2=(bc+cb)\in R^D$ by Lemma 2.1(2). It follows that $b+ c \in R^D$.

$(2)\Leftrightarrow (3)$ Its proof is similar to $(1)\Leftrightarrow(2)$.
\end{proof}

\begin{lemma} Let $a\in R$ with  $a-a^2\in R^D$ or $a+a^2\in R^D$. Then $a\in R^D$.
\end{lemma}
\begin{proof} We only need to prove the situation when $a-a^2\in R^D$ with $x=(a-a^2)^D$.

By Lemma 2.1(1), it is clear $ax=xa$ since $a(a-a^2)=(a-a^2)a$.

Since $a-a^2\in R^D$, we get $(a-a^2)^n=(a-a^2)^{n+1}x$ for some integer $n\geq 1$, that is,
\begin{center}
$a^n(1-a)^n=a^{n+1}(1-a)^{n+1}x$.
\end{center}
Note that $$a^n(1-a)^n=a^n(1+\sum_{i=1}^nC_n^i(-a)^i).$$ It follows that
$$a^n=a^{n+1}[(1-a)^{n+1}x+\sum_{i=1}^nC_n^i(-a)^{i-1}]=[(1-a)^{n+1}x+\sum_{i=1}^nC_n^i(-a)^{i-1}]a^{n+1}.$$
This shows $a^n\in a^{n+1}R\cap Ra^{n+1}$. Hence, $a\in R^D$.
\end{proof}
\section{\bf Main results}
~~~~In what follows, $p$ and $q$ always mean two arbitrary idempotents in a ring $R$. We give some equivalent conditions for the Drazin invertibility of $p-q$, $pq$, $pq-qp$ and $pq+qp$.

\begin{proposition} The following statements are equivalent:\\
$(1)$ $1-pq\in R^D$,
$(2)$ $p-pq\in R^D$,
$(3)$ $p-qp\in R^D$,
$(4)$ $1-pqp\in R^D$,
$(5)$ $p-pqp\in R^D$,\\
 $(6)$ $1-qp\in R^D$,
$(7)$ $q-qp\in R^D$,
$(8)$ $q-pq\in R^D$,
$(9)$ $1-qpq\in R^D$,
$(10)$ $q-qpq\in R^D$.
\end{proposition}
\begin{proof}  $(1)\Leftrightarrow(6)$ is obvious by Lemma 2.4. We only need to prove that $(1)-(5)$ are equivalent.

$(1)\Leftrightarrow(4)$ It is clear that $1-pq=1-p(pq)$. Thus, $1-pq$ is Drazin invertible if and only if $1-pqp$ is Drazin invertible by Lemma 2.4.

$(4)\Rightarrow(5)$ Since $p\in R^D$ and $p(1-pqp)=(1-pqp)p=p-pqp$, we obtain that $p-pqp$ is Drazin invertible according to Lemma 2.2.

$(5)\Rightarrow(4)$  Suppose $a=p-pqp$, $b=1$. Then $a$ and $b$ are Drazin invertible. Since $1-pqp=(p-pqp)+1-p$, it follows that $1-pqp=ap+b(1-p)$ is Drazin invertible in view of Lemma 2.5.

$(2)\Leftrightarrow(5)$ Since $p-pq=pp(1-q)$ and $p-pqp=p(1-q)p$, the result follows by Lemma 2.3.

$(2)\Leftrightarrow(3)$ It is evident according to Lemma 2.3.
\end{proof}

We replace $p$ and $q$ by $1-p$ and $1-q$ respectively in Proposition 3.1 to get the following result.

\begin{corollary}  The following statements are equivalent:\\
$(1)$ $p+q-pq\in R^D$, ~~~~~~~~~~~$(2)$ $q-pq\in R^D$, ~~~~~~~~~~~$(3)$ $q-qp\in R^D$,\\
$(4)$ $p+(1-p)(q-qp)\in R^D$, $(5)$ $(1-p)q(1-p)\in R^D$, ~$(6)$ $p+q-qp\in R^D$,\\
$(7)$ $p-qp\in R^D$, ~~~~~~~~~~~~~~~$(8)$ $p-pq\in R^D$,~~~~~~~~~~~~~$(9)$ $q+(1-q)(p-pq)\in R^D$,\\
 $(10)$ $(1-q)p(1-q)\in R^D$.
\end{corollary}

\begin{theorem} The following statements are equivalent:\\
$(1)$ $p-q\in R^D$,\\
$(2)$ $1-pq\in R^D$,\\
$(3)$ $p+q-pq\in R^D$.
\end{theorem}

\begin{proof} $(1)\Rightarrow(2)$ It is easy to check that $p(p-q)^2=(p-q)^2p=p-pqp$. Hence,
\begin{center}
$1-pqp=(p-q)^2p+1-p$.
\end{center}
Let $a=(p-q)^2$ and $b=1$. Then $ap=pa$, $bp=pb$. Since $p-q\in R^D$, we obtain that $a=(p-q)^2\in R^D$. In view of Lemma 2.5, $1-pqp=ap+b(1-p)\in R^D$. Therefore, $1-pq$ is Drazin invertible by Lemma 2.4.

$(2)\Leftrightarrow(3)$ is clear by Proposition 3.1 and Corollary 3.2.

$(3)\Rightarrow(1)$ Let $a=1-pqp$, $b=1-(1-p)(1-q)(1-p)$. Then we get $ap=pa$ and $bp=pb$. Since
 \begin{center}
 $1-(1-p)(1-p)(1-q)=p+q-pq \in R^D$,
\end{center}
we obtain that $b\in R^D$ by Lemma 2.4 and $a\in R^D$ by Proposition 3.1 and Corollary 3.2. Note that $(p-q)^2=ap+b(1-p)$. We get $(p-q)^2\in R^D$ by Lemma 2.5. Hence, $p-q\in R^D$.
\end{proof}

\begin{theorem}  The following statements are equivalent:\\
$(1)$ $pq \in R^D$, \\
$(2)$ $1-p-q \in R^D$, \\
$(3)$ $(1-p)(1-q)\in R^D$.
\end{theorem}

\begin{proof}
$(1)\Leftrightarrow(2)$ Let $p_1=1-p$ and $q_1=q$. Then $p_1-q_1\in R^D$ if and only if $q_1-p_1q_1\in R^D$ by Proposition 3.1 and Theorem 3.3. Since $p_1-q_1=1-p-q$ and $q_1-p_1q_1=pq$, $(1)\Leftrightarrow(2)$ holds.

$(1)\Leftrightarrow(3)$ Set $p_1=1-p$, $q_1=q$. Then $p_1-p_1q_1\in R^D$ if and only if $q_1-p_1q_1\in R^D$ by Proposition 3.1. Since $p_1-p_1q_1=(1-p)(1-q)$ and $q_1-p_1q_1=pq$, the result follows.
\end{proof}

\begin{theorem}  The following statements are equivalent:\\
$(1)$ $pq-qp\in R^D$,\\
$(2)$ $pq\in R^D$ and $p-q\in R^D$.
\end{theorem}

\begin{proof} Suppose $b=pq(1-p)$ and $c=(1-p)qp$. It follows that $b-c=pq-qp$.

$(1)\Rightarrow(2)$ By hypothesis $b-c\in R^D$, we obtain $pqp-(pqp)^2=pq(1-p)qp=bc\in R^D$ by Lemma 2.6. It follows that $pqp\in R^D$ by Lemma 2.7. Hence, $pq\in R^D$. Thus, $pq-qp\in R^D$ implies $pq\in R^D$.

Similarly,
$$(1-p)q-q(1-p)=-(pq-qp)\in R^D$$
implies $q-pq=(1-p)q \in R^D$. Therefore, $p-q\in R^D$ by Proposition 3.1 and Theorem 3.3.

$(2)\Rightarrow(1)$ By Lemma 2.3, both $pqp$ and $(p-q)^2$ are Drazin invertible. Note that $bc=pq(1-p)qp=pqp(p-q)^2=(p-q)^2pqp$. It follows that $bc$ is Drazin invertible by Lemma 2.2. Hence, according to Lemma 2.6, we have $pq-qp=b-c\in R^D$.
\end{proof}

\begin{theorem}  The following statements are equivalent:\\
$(1)$ $pq+qp\in R^D$,\\
$(2)$ $pq\in R^D$ and $p+q\in R^D$.
\end{theorem}
\begin{proof} $(1)\Rightarrow(2)$ Since $pq+qp=-(p+q)+(p+q)^2=(p+q-1)+(p+q-1)^2\in R^D$, $p+q\in R^D$ and $p+q-1\in R^D$ according to Lemma 2.7. Therefore, $pq\in R^D$ by Theorem 3.4.

$(2)\Rightarrow(1)$ Since $pq+qp=(p+q)(p+q-1)=(p+q-1)(p+q)$, $pq+qp\in R^D$ by Lemma 2.2 and Theorem 3.4.

\end{proof}
\begin{remark}
{\rm Let $p$, $q$ be two idempotents in a Banach algebra. Then, $p+q$ is Drazin invertible if and only if $p-q$ is Drazin invertible [10]. However, in general, this need not be true in a ring. For example, let $R=\mathbb{Z}$, $p=q=1$. Then $p-q=0$ is Drazin invertible, but $p+q=2$ is not Drazin invertible.}
\end{remark}

\centerline {\bf ACKNOWLEDGMENTS} This research is supported by the National Natural Science Foundation of China (10971024),
the Specialized Research Fund for the Doctoral Program of Higher Education (20120092110020), the Natural Science Foundation of Jiangsu Province (BK2010393) and the Foundation of Graduate
Innovation Program of Jiangsu Province(CXLX13-072).

\end{document}